\documentclass[12pt]{amsart}
\usepackage{amsfonts}
\usepackage{ifthen}
\usepackage{verbatim}
\usepackage{amsthm}
\usepackage{amsmath}
\usepackage{amssymb}
\usepackage{graphicx}
\usepackage{subfigure}
\usepackage{amscd,amssymb,amsthm}
\usepackage{color,xcolor}

\newcounter{minutes}
\divide\time by 60
\newcounter{hours}
\multiply\time by 60 \addtocounter{minutes}{-\time}

\title{Global properties of the symmetrized $S$-divergence}

\author{Slavko Simi\'c}

\address{ Mathematical Institute SANU, Kneza Mihaila 36, 11000
Belgrade, Serbia} \email{ ssimic@turing.mi.sanu.ac.rs}


\newtheorem{lemma}[equation]{Lemma}
\newtheorem{proposition}[equation]{Proposition}

\newtheorem{corollary}[equation]{Corollary}
\newtheorem{remark}[equation]{Remark}

\newcommand{\beq}{\begin{equation}}
\newcommand{\eeq}{\end{equation}}

\numberwithin{equation}{section}

\begin{document}

\def\thefootnote{}
\footnotetext{ \texttt{\tiny File:~\jobname .tex,
          printed: \number\year-\number\month-\number\day,
          \thehours.\ifnum\theminutes<10{0}\fi\theminutes}
} \makeatletter\def\thefootnote{\@arabic\c@footnote}\makeatother


\maketitle

\begin{abstract}
In this paper we give a study of the symmetrized divergences
$U_s(p,q)=K_s(p||q)+K_s(q||p)$ and $V_s(p,q)=K_s(p||q)K_s(q||p)$,
where $K_s$ is the relative divergence of type $s, s\in\mathbb R$.
Some basic properties as symmetry, monotonicity and log-convexity
are established. An important result from the Convexity Theory is
also proved.
\end{abstract}


\section{Introduction}

\vspace{0.5cm}

Let
$$
\Omega^+=\{p=\{p_i\} \ | \ p_i>0, \sum p_i=1\},
$$
be the set of finite discrete probability distributions.

\vspace{0.5cm}

 One of the most general probability measures which
is of importance in Information Theory is the famous Csisz\'{a}r's
$f$-divergence $C_f(p||q)$ ([5]), defined by

\vspace{0.5cm}

{\bf Definition 1} {\it For a convex function $f: (0,\infty)\to
\mathbb R$, the $f$-divergence measure  is given by
$$
C_f(p||q):=\sum q_i f(p_i/q_i),
$$
where $p, q\in \Omega^+$.}

\vspace{0.5cm}

 Some important information measures are just particular cases of the Csisz\'{a}r's $f$-divergence.

\vspace{0.5cm}

For example,

\vspace{0.5cm}

(a) \ taking $f(x)=x^\alpha, \ \alpha>1$, we obtain the
$\alpha$-order divergence defined by
$$
I_\alpha (p||q):=\sum p_i^\alpha q_i^{1-\alpha};
$$

\vspace{0.5cm}

{\bf Remark} \ {\it The above quantity is an argument in
well-known theoretical divergence measures such as Renyi
$\alpha$-order divergence  $I^R_\alpha(p||q)$ or Tsallis
divergence $I^T_\alpha(p||q)$, defined as}
$$
I^R_\alpha(p||q):=\frac{1}{\alpha-1}\log I_\alpha(p||q); \ \
I^T_\alpha(p||q):=\frac{1}{\alpha-1}(I_\alpha(p||q)-1).
$$

\vspace{0.5cm}

 (b) \ for $f(x)=x\log x$, we obtain the Kullback-Leibler divergence ([3]) defined by
$$
K(p||q):=\sum p_i\log(p_i/q_i);
$$

\vspace{0.5cm}

(c) \ for $f(x)=(\sqrt x -1)^2$, we obtain the Hellinger distance
$$
H^2(p, q):=\sum (\sqrt p_i-\sqrt q_i)^2;
$$

\vspace{0.5cm}

(d) \ if we choose $f(x)=(x-1)^2$, then we get the
$\chi^2$-distance
$$
\chi^2(p, q):=\sum (p_i-q_i)^2/q_i.
$$

\vspace{0.5cm}

The generalized measure $K_s(p||q)$, known as {\it the relative
divergence of type $s$} ([6], [7]), or simply {\it
$s$-divergence}, is defined by

\vspace{0.5cm}

$$
K_s(p||q):=\begin{cases} (\sum p_i^s q_i^{1-s}-1)/ s(s-1) &, s\in \mathbb R/\{0, 1\};\\
                   K(q||p)                         &, s=0;\\
                   K(p||q)                         &, s=1.
                   \end{cases}
                   $$

\vspace{0.5cm}

It include the Hellinger and $\chi^2$ distances as particular
cases.

\vspace{0.5cm}

Indeed,
$$
K_{1/2}(p||q)=4(1-\sum\sqrt {p_i q_i})=2\sum(p_i+q_i-2\sqrt {p_i
q_i})=2H^2(p,q);
$$
$$
K_2(p||q)=\frac{1}{2}(\sum \frac{p_i^2}{q_i}-1)=\frac{1}{2}\sum
\frac{(p_i-q_i)^2}{q_i}=\frac{1}{2}\chi^2(p, q).
$$

\vspace{0.5cm}

The $s$-divergence represents an extension of Tsallis divergence
to the real line and accordingly is of importance in Information
Theory. Main properties of this measure are given in \cite{t}.

\vspace{0.5cm}

{\bf Theorem A} {\it For fixed $p, q\in\Omega^+, p\neq q$, the
$s$-divergence is a positive, continuous and convex function in
$s\in\mathbb R$.}

\vspace{0.5cm}

We shall use in this article a stronger property.

\vspace{0.5cm}

{\bf Theorem B} {\it For fixed $p, q\in\Omega^+, p\neq q$, the
$s$-divergence is a log-convex function in $s\in\mathbb R$.}

\vspace{0.5cm}

\begin{proof} This  is a corollary of an assertion proved in
\cite{ss}. It says that for arbitrary positive sequence $\{x_i\}$
and associated weight sequence $q\in Q$ (see Appendix), the
quantity $\lambda_s$ defined by
$$
\lambda_s:=\frac{\sum q_ix_i^s-(\sum q_ix_i)^s}{s(s-1)}
$$

is logarithmically convex in $s\in\mathbb R$.

\vspace{0.5cm}

Putting there $x_i=p_i/q_i$, we obtain that $\lambda_s=K_s(p||q)$
is log-convex in $s\in\mathbb R$. Hence, for any real $s,t$ we
have that

$$
K_s(p||q)K_t(p||q)\ge K_{\frac{s+t}{2}}^2(p||q).
$$

\end{proof}

\vspace{0.5cm}

Among all mentioned measures, only Hellinger distance has a
symmetry property $H^2=H^2(p,q)=H^2(q,p)$. Our aim in this paper
is to investigate some global properties of the symmetrized
measures $U_s=U_s(p,q)=U_s(q,p):=K_s(p||q)+K_s(q||p)$ and
$V_s=V_s(p,q)=V_s(q,p):=K_s(p||q)K_s(q||p)$. Since S. Kullback and
R. Leibler themselves in their fundamental paper \cite{kl} (see
also \cite{j}) worked with the symmetrized variant
$J(p,q):=K(p||q)+K(q||p)=\sum(p_i-q_i)\log(p_i/q_i)$, our results
can be regarded as a continuation of their ideas.

\vspace{0.5cm}

\section{Results and Proofs}

\vspace{0.5cm}

We shall give firstly some properties of the symmetrized
divergence $V_s=K_s(p||q)K_s(q||p)$.

\vspace{0.5cm}

\begin{proposition} \ 1. For arbitrary, but fixed probability distributions
$p,q\in\Omega^+, p\neq q$, the divergence $V_s$ is a positive and
continuous function in $s\in\mathbb R$.

2. $V_s$ is a log-convex (hence convex) function in $s\in\mathbb
R$.

3. The  graph of $V_s$ is symmetric with respect to the line
$s=1/2$, bounded from below with the universal constant $4H^4$ and
unbounded from above.

4. $V_s$ is monotone decreasing for $s\in (-\infty, 1/2)$ and
monotone increasing for $s\in (1/2,+\infty)$.

5. The inequality

$$
V_s^{t-r}\le V_r^{t-s}V_t^{s-r}
$$

holds for any $r<s<t$.

\end{proposition}

\vspace{0.5cm}

\begin{proof} The Part 1. is a simple consequence of Theorem A
above.

\vspace{0.5cm}

The proof of Part 2. follows by using Theorem B. Namely, for any
$s,t\in\mathbb R$ we have

$$
V_sV_t=[K_s(p||q)K_s(q||p)][K_t(p||q)K_t(q||p)]=[K_s(p||q)K_t(p||q)][K_s(q||p)K_t(q||p)]
$$
$$
\ge
[K_\frac{s+t}{2}(p||q)]^2[K_\frac{s+t}{2}(q||p)]^2=[V_\frac{s+t}{2}]^2.
$$

\vspace{0.5cm}

3. Note that

$$
K_s(p||q)=K_{1-s}(q||p); K_s(q||p)=K_{1-s}(p||q).
$$

\vspace{0.5cm}

Hence $V_s=V_{1-s}$, that is $V_{1/2-s}=V_{1/2+s}, s\in\mathbb R$.

\vspace{0.5cm}

Also,

$$
V_s=K_s(p||q)K_s(q||p)=K_s(p||q)K_{1-s}(p||q)\ge
K_{1/2}^2(p||q)=4H^4.
$$

\vspace{0.5cm}

4. We shall prove only the "increasing" assertion. The other part
follows from graph symmetry.

\vspace{0.5cm}

Therefore, for any $1/2<x<y$ we have that

$$
1-y<1-x<x<y.
$$

Applying Proposition X (see Appendix) with $a=1-y, b=y, s=1-x,
t=x; f(s):=\log K_s(p||q)$, we get

$$
\log K_x(p||q)+\log K_{1-x}(p||q)\le \log K_y(p||q)+\log
K_{1-y}(p||q),
$$

\vspace{0.5cm}

that is $V_x\le V_y$ for $x<y$.

\vspace{0.5cm}

5. From the parts 1 and 2, it follows that $\log V_s$ is a
continuous and convex function on $\mathbb R$. Therefore we can
apply the following alternative form \cite{hlp}:

\vspace{0.5cm}

\begin{lemma} \label{3} If $\phi(s)$ is continuous and convex for all
$s$ of an open interval $I$ for which $s_1<s_2<s_3$, then

$$
\phi(s_1)(s_3-s_2)+\phi(s_2)(s_1-s_3)+\phi(s_3)(s_2-s_1)\ge 0.
$$
\end{lemma}

\vspace{0.5cm}

Hence, for $r<s<t$ we get
$$
(t-r)\log V_s\le (t-s)\log V_r+(s-r)\log V_t,
$$
which is equivalent to the assertion of Part 5.
\end{proof}

\vspace{0.5cm}

Properties of the symmetrized measure $U_s:=K_s(p||q)+K_s(q||p)$
are very similar; therefore some analogous proofs will be omitted.

\vspace{0.5cm}

\begin{proposition} \ 1. The divergence $U_s$ is a positive and continuous function in
$s\in\mathbb R$.

2. $U_s$ is a log-convex function in $s\in\mathbb R$.

3. The  graph of $U_s$ is symmetric with respect to the line
$s=1/2$, bounded from below with $4H^2$ and unbounded from above.

4. $U_s$ is monotone decreasing for $s\in (-\infty, 1/2)$ and
monotone increasing for $s\in (1/2,+\infty)$.

5. The inequality

$$
U_s^{t-r}\le U_r^{t-s}U_t^{s-r}
$$

holds for any $r<s<t$.

\end{proposition}

\begin{proof} \ 1. Omitted.

\vspace{0.5cm}

2. Since both $K_s$ and $V_s$ are log-convex functions, we get

$$
U_sU_t-U_{\frac{s+t}{2}}^2
$$
$$
=[K_s(p||q)+K_s(q||p)][K_t(p||q)+K_t(q||p)]-[K_{\frac{s+t}{2}}(p||q)+K_{\frac{s+t}{2}}(q||p)]^2
$$
$$
=[K_s(p||q)K_t(p||q)-K_{\frac{s+t}{2}}(p||q)^2]+[K_s(q||p)K_t(q||p)-K_{\frac{s+t}{2}}(q||p)^2]
$$
$$
+
[K_s(p||q)K_t(q||p)+K_s(q||p)K_t(p||q)-2K_{\frac{s+t}{2}}(p||q)K_{\frac{s+t}{2}}(q||p)]
$$
$$
\ge
[K_s(p||q)K_t(p||q)-K_{\frac{s+t}{2}}(p||q)^2]+[K_s(q||p)K_t(q||p)-K_{\frac{s+t}{2}}(q||p)^2]
$$
$$
+2[\sqrt{V_sV_t}-V_{\frac{s+t}{2}}]\ge 0.
$$

\vspace{0.5cm}

3. The graph symmetry follows from the fact that $U_s=U_{1-s},
s\in \mathbb R$.

\vspace{0.5cm}

We also have

$$
U_s\ge 2\sqrt{V_s}\ge 4H^2.
$$

\vspace{0.5cm}

Finally, since $p\neq q$ yields $\max\{p_i/q_i\}=p_*/q_*>1$,

we get

$$
K_s(p||q)>\frac{q_*(p_*/q_*)^s-1}{s(s-1)}\to\infty \ (s\to\infty).
$$

It follows that both $U_s$ and $V_s$ are unbounded from above.

\vspace{0.5cm}

4. Omitted.

\vspace{0.5cm}

5. \ The proof is obtained by another application of Lemma \ref{3}
with $\phi(s)=\log U_s$.

\end{proof}

\vspace{0.5cm}

\begin{remark} We worked here with the class $\Omega^+$ for the sake of
simplicity. Obviously that all results hold, after suitable
adjustments, for arbitrary probability distributions and in the
continuous case as well.

\end{remark}

\vspace{0.5cm}

\begin{remark} \ It is not difficult to see that the same properties
are valid for normalized divergences
$U_s^*=\frac{1}{2}(K_s(p||q)+K_s(q||p))$ and
$V_s^*=\sqrt{K_s(p||q)K_s(q||p)}$, with

$$
2H^2\le V_s^*\le U_s^*.
$$

\end{remark}

\vspace{0.5cm}

\section{Appendix}

\vspace{0.5cm}

\centerline{\bf A convexity property}

\vspace{0.5cm}

Most general class of convex functions is defined by the
inequality

\begin{equation}\label{1}
\frac{\phi(x)+\phi(y)}{2}\ge \phi(\frac{x+y}{2}).
\end{equation}

\vspace{0.5cm}

A function which satisfies this inequality in a certain closed
interval $I$ is called {\it convex} in that interval.
Geometrically it means that the midpoint of any chord of the curve
$y=\phi(x)$ lies above or on the curve.

\vspace{0.5cm}

Denote now by $Q$ the family of {\it weights} i.e., positive real
numbers summing to $1$. If $\phi$ is continuous, then much more
can be said i.e., the inequality

\begin{equation}\label{2}
p\phi(x)+q\phi(y)\ge\phi(px+qy)
\end{equation}

holds for any $p,q\in Q$. Moreover, the equality sign takes place
only if $x=y$ or $\phi$ is linear (cf. \cite{hlp}).

\vspace{0.5cm}

We shall prove here an interesting property of this class of
convex functions.

\vspace{0.5cm}

{\bf Proposition X} {\it Let $f(\cdot)$ be a continuous convex
function defined on a closed interval $[a, b]:=I$. Denote
$$
F(s,t):=f(s)+f(t)-2f(\frac{s+t}{2}).
$$

\vspace{0.5cm}

Then}
$$
\max_{s, t\in I} F(s,t)=F(a, b).\eqno(1)
$$

\vspace{0.5cm}

\begin{proof}

\vspace{0.5cm}

It suffices to prove that the inequality
$$
F(s,t)\le F(a, b)
$$
holds for $a<s<t<b$.

\vspace{0.5cm}

 In the sequel we need the following
assertion (which is of independent interest).

\vspace{0.5cm}

 \begin{lemma}  Let $f(\cdot)$ be a continuous convex
function on some interval $I\subseteq \mathbb R$. If $x_1, x_2,
x_3\in I$ and $x_1<x_2<x_3$, then

\vspace{0.5cm}

$$
(i) \ \ \ \frac{f(x_2)-f(x_1)}{2}\le f(\frac{x_2+x_3}
{2})-f(\frac{x_1+x_3}{2});
$$
$$
(ii) \ \ \ \frac{f(x_3)-f(x_2)}{2}\ge f(\frac{x_1+x_3}
{2})-f(\frac{x_1+x_2}{2}).
$$
\end{lemma}

 \begin{proof}

\vspace{0.5cm}

We shall prove the first part of the lemma; the proof of second
part goes along the same lines.

\vspace{0.5cm}

 Since  $x_1<x_2<\frac{x_2+x_3}{2}< x_3$, there exist
$p, q; \ 0< p, q< 1, p+q=1$ such that
$x_2=px_1+q\frac{x_2+x_3}{2}$.

\vspace{0.5cm}

 Hence,
$$
\frac{f(x_1)-f(x_2)}{2} + f(\frac{x_2+x_3}{2})\ge\frac{1}
{2}[f(x_1)-(pf(x_1)+qf(\frac{x_2+x_3}{2}))]+ f(\frac{x_2+x_3}{2})
$$
$$
=\frac{q}{2} f(x_1)+\frac{2-q}{2} f(\frac{x_2+x_3}{2})\ge
f(\frac{q}{2} x_1+\frac{2-q}{2}(\frac{x_2+x_3}
{2}))=f(\frac{x_1+x_3}{2}).
$$
\end{proof}

\vspace{0.5cm}

 Now, applying the part (i) with $x_1=a, x_2=s,
x_3=b$ and the part (ii) with $x_1=s, x_2=t, x_3=b$, we get
$$
\frac{f(s)-f(a)}{2}\le f(\frac{s+b}{2})-f(\frac{a+b}{2}); \eqno
(2)
$$
$$
\frac{f(b)-f(t)}{2}\ge f(\frac{s+b}{2})-f(\frac{s+t}{2}), \eqno
(3)
$$
respectively.

\vspace{0.5cm}

 Subtracting (2) from (3), the desired inequality
follows.

\end{proof}

\vspace{0.5cm}

\begin{corollary} Under the conditions of Proposition X, we
have that the double inequality

$$
2f(\frac{a+b}{2})\le f(t)+f(a+b-t)\le f(a)+f(b)\eqno(4)
$$

holds for each $t\in I$.

\end{corollary}

\begin{proof}

Since the condition $t\in I$ is equivalent with $a+b-t\in I$,
applying Proposition X with $s=a+b-t$ we obtain the right-hand
side of (4). The left-hand side inequality is obvious.
\end{proof}

\vspace{0.5cm}

\begin{remark} The relation (4) is a kind of pre-Hermite-Hadamard
inequalities. Indeed, integrating both sides of (4) over $I$, we
obtain the famous H-H inequality

$$
f(\frac{a+b}{2})\le\frac{1}{b-a}\int_a^b
f(t)dt\le\frac{f(a)+f(b)}{2},
$$

since $\int_a^b f(a+b-t)dt=\int_a^b f(t)dt$.
\end{remark}


\vspace{0.5cm}

\begin{thebibliography}{HIMPS}
\footnotesize

\bibitem [CS]{cs}{\sc Csisz\'{a}r, I.}, Information-type measures of difference of probability functions and indirect
observations, { Studia Sci. Math. Hungar.} 2 (1967), \ 299-318.

\bibitem [HLP]{hlp}{\sc G.H. Hardy, J.E. Littlewood and G. Polya}, {\em
    Inequalities}, Cambridge University Press, Cambridge, 1978.



\bibitem[J]{j} {\sc Jeffreys, H.}, An invariant form for the prior probability
in estimation problems, Proc. Roy. Soc. Lon., Ser. A, {\bf 186}
(1946), 453-461. .

 \bibitem[KL]{kl} {\sc Kullback, S. and Leibler, R. A.}, On information and sufficiency, {Ann. Math. Stat. } {\bf 22}(1) (1951), 79--86.


\bibitem[K]{k}
{\sc Kullback, S.}, {\em Information Theory and Statistics}, John
Willey \& Sons, New York (1959).
\bibitem[SS]{ss} {\sc Simic, S.}, On logarithmic convexity for differences of
power means, J. Inequal. Appl. Article ID 37359 (2007), 8 p.


  \bibitem[S] {s}{\sc Simic, S.}, {On a new moment inequality},
  Statist. Probab. Lett. {\bf 78}(16) (2008), 2671--2678.


\bibitem[T]{t} {\sc Taneja, I.J.}, New developments in generalized information
measures, Advances in Imaging and Electron Physics, {\bf 91}
(1995), 37-135.

\bibitem[V]{v} {\sc Vajda, I.}, {\em Theory of Statistical Inference and
Information}, Kluwer Academic Press, London (1989).





\end{thebibliography}
\end{document}